\documentclass[journal]{IEEEtran}
\ifCLASSINFOpdf
\else
\fi
\usepackage[all]{xy}
\usepackage{mathrsfs,amsbsy,color,graphicx,amsmath,amsfonts,algorithm,algorithmic,latexsym}
\usepackage{subfigure}
\usepackage{amsthm}

\usepackage{algorithmic}

\usepackage{array}

\ifCLASSOPTIONcompsoc
  \usepackage[caption=false,font=normalsize,labelfont=sf,textfont=sf]{subfig}
\else
  \usepackage[caption=false,font=footnotesize]{subfig}
\fi
\usepackage{fixltx2e}

\usepackage{stfloats}

\newtheorem{proposition}{Proposition}

\begin{document}
%
\title{On the Existence and Computation\\ of Minimum Attention Optimal
Control Laws}
%
%
%

\author{Pilhwa Lee and ~F.C. Park,~\IEEEmembership{Fellow}
\thanks{Pilhwa Lee (lead author) is with the Department of Mathematics,
Morgan State University, Baltimore, MD, USA,
pilhwa.lee@morgan.edu.}
\thanks{F.C. Park (corresponding author) is with the Department of
Mechanical Engineering, Seoul National University, Seoul, Korea,
fcp@snu.ac.kr.}
\thanks{Manuscript received Nov 22, 2019; revised January 20, 2021.}}

%
%

\markboth{IEEE Transactions on Automatic Control,~Vol.~xx, No.~xx, xx~xxx}%
{Shell \MakeLowercase{\textit{et al.}}: Bare Demo of IEEEtran.cls for IEEE Journals}
%



\maketitle

\begin{abstract}
One means of capturing the cost of control implementation of a general
nonlinear control system is via Brockett's minimum attention criterion,
defined as a multidimensional integral of the rate of change of the
control with respect to state and time.  Although shown to be important
in human motor control and robotics applications, a practical difficulty
with this criterion is that the existence of solutions is not always
assured; even when they exist, obtaining local solutions numerically is
difficult.  In this paper we prove that, for the class of controls
consisting of the sum of a time-varying feedforward term and a
time-varying feedback term linear in the state, existence of a suboptimal
solution can be guaranteed.  We also derive a provably convergent gradient
descent algorithm for obtaining a local solution, by appealing to the
Liouville equation representation of a nonlinear control system and adapting
iterative methods originally developed for boundary flow control.  Our
methodology is illustrated with a two degree-of-freedom planar robot example.
\end{abstract}


%
\IEEEpeerreviewmaketitle


\section{Introduction}

Given a general nonlinear system $\dot{x} = f(x,u,t)$ evolving on the state
space ${\cal X}$, where $x \in {\cal X} \subseteq \mathbb{R}^n$ is the state
and $u \in {\cal U} \subseteq \mathbb{R}^m$ is the control, Brockett
\cite{Brockett} proposes the following functional---referred to as the
{\bf minimum attention criterion}---as a measure of the cost of
implementation of a control law $u = u(x,t)$:
\begin{equation}
J(u) = \frac{1}{2} \int_0^{T} \int_{\cal X}
 \left\| \frac{\partial u}{\partial x} \right\|^2
+ \left\| \frac{\partial u}{\partial t} \right\|^2 dx \: dt.
\label{eqn:attention_brockett}
\end{equation}
The basic premise behind this criterion is that the simplest
control law to implement is a constant input: the more frequently a control
changes, the more resources are required to implement it. Since control laws
typically depend on both the state and time, their cost of implementation can
be linked with the rate at which the control varies with respect to changes
in state -- the first term of (\ref{eqn:attention_brockett}) -- and changes
in time -- the second term of (\ref{eqn:attention_brockett}).

Some intriguing connections have been pointed out between the attention
criterion and existing theories for human motor control and learning
\cite{Jang, Todorov, Schall, Wolpert, Battaglia-mayer,Nakano, Kawato,
Wolpert2,Liu}.  For example, the minimum attention control paradigm
offers a generative explanatory model for human motor skill learning, in
which skills that are initially reliant on feedback gradually become more
and more open loop with practice, e.g., perfecting a golf swing.  At the same
time, the minimum attention principle can also explain how, in tasks that
rely on sensory feedback like catching a ball, the reliance on visual
feedback is initially low, but steadily increases and peaks immediately
before the ball is caught (we address this example further in our
experiments).  New insights and perspectives on the control of soft robots
have also been obtained via the minimum attention paradigm \cite{Santina}.

The main challenge in forging the minimum attention paradigm into
a practical control methodology is that solutions to the associated
multidimensional variational problem for~(\ref{eqn:attention_brockett})
are difficult to come by; even the existence of solutions is not always
guaranteed in the general case.

One way to make this problem more tractable is to restrict attention
to systems of the form $\dot{x} = f(x) + G(x)u$, with $G(x) \in
\mathbb{R}^{n \times m}$ given and $u$ of the form $u(x,t) =
K(t) x + v(t)$, where $K(t) \in \mathbb{R}^{m \times n}$
is a feedback gain matrix and $v(t) \in \mathbb{R}^m$ is a feedforward
term.  For this class of controls, some results on minimum attention
trajectories have been obtained for simple kinematic models of wheeled
mobile robots \cite{sohee_ark}.  Additional assumptions on the attention
functional, such as setting the terminal cost to be the mean-square state
error and weighting the attention functional probabilistically, make
it possible to approximate the problem as a linear quadratic regulator
(LQR), for which efficient solutions are readily obtainable
\cite{attention_brockett}, \cite{Jang}.  By and large, however, the minimum
attention paradigm has been effectively applied only in limited cases,
with only approximate solutions obtained under restricted settings.  

In this paper we expand the class of problems for which minimum attention
control laws can be meaningfully framed, in the sense of guaranteeing the
existence of solutions and obtaining them numerically. The main
contributions are (i) a proof of the existence of a minimum
attention control under the assumption of a control of the form $u(x,t)
= K(t) x + v(t)$ over a bounded state space, with $K(t)$ and $v(t)$
twice continuously differentiable, and (ii) an iterative gradient
descent algorithm for computing this minimum attention control law.

The key ideas that we draw upon are the Liouville partial differential
equation (PDE) representation for nonlinear systems as proposed by Brockett
\cite{Brockett2}, and the observation that the existence question for the
minimum attention problem is structurally similar to the well-posedness
problem for boundary flow control in the Navier-Stokes equations
\cite{Gunzburger}. Specifically, \cite{Gunzburger} proves the existence
of a solution minimizing the rate of change in the control with respect
to space and time for a two-dimensional boundary fluid flow, while tracking
prescribed fluid flow velocities governed by the Navier-Stokes equations.
We show that the minimum attention control problem posed in the Liouville 
PDE setting is, despite some differences (e.g., the state space for
the minimum attention problem is of arbitrary dimension $n$ rather than
two-dimensional, and a control over the entire state space is sought
rather than only on the boundary), structurally similar to the flow
control problem addressed in \cite{Gunzburger}.

Addressing these differences, we show that the one-shot iterative 
method developed in \cite{Gunzburger} can be generalized to the minimum
attention problem. Specifically, an iterative Monte Carlo-based gradient
descent method is developed to find local solutions to the minimum attention
problem, whose convergence to a local minimum can also be guaranteed via
ellipticity (specifically, boundedness of the second-order variation of
the associated Lagrangian).

The remainder of the paper is organized as follows.  Section
\ref{sec:formulation} formulates the minimum attention problem in a
Liouville PDE setting and makes explicit the structural similarities
and differences with the flow control problem of \cite{Gunzburger}.
Section~\ref{sec:proof_and_iterative_algorithm} proves the existence
of a solution for the class of controls considered in this paper,
and describes the iterative gradient descent method together with
a proof of its convergence to a local solution. Section
\ref{sec:example} provides examples of minimum attention control
laws obtained for a two-link planar robot.


\section{Problem Formulation}
\label{sec:formulation}

Given the first-order nonlinear system
\begin{equation}
\dot{x} = f(x,u,t),
\label{eqn:xdotfxu}
\end{equation}
where $x \in \mathcal{X} \subseteq \mathbb{R}^{n}$ is the state,
$u \in \mathcal{U} \subseteq \mathbb{R}^{m}$ is the control, and $f$ is
continuously differentiable with respect to $x$ and $u$, the standard 
optimal control problem for~(\ref{eqn:xdotfxu}) is usually formulated as
seeking the optimal $u^*$ that minimizes the criterion
\begin{equation}
\phi(x(T)) + \int_0^T L(x,u) \: dt,
\label{eqn:intLxudt}
\end{equation}
subject to boundary conditions on the initial and terminal states and
other equality and inequality constraints imposed on the state and control.
In practical situations one may also have to contend with a distribution of
initial states rather than a single initial state. In \cite{Brockett2} the
case is made that by replacing (\ref{eqn:xdotfxu}) by the Liouville partial
differential equation 
\begin{equation}
\frac{\partial \rho(x,t)}{\partial t}
= -\left< \frac{\partial}{\partial x}, f(x,u) \rho(x,t) \right>,
\label{eqn:Liouville_evol}
\end{equation}
where $\rho(x,t)$ is the probability density for $x$ with boundary
condition $\rho(x,0) = \rho_0(x)$, then not only can an initial distribution
of states be naturally included, but the criterion~(\ref{eqn:intLxudt}) can 
also be expanded to include, e.g., non-trajectory dependent terms such as
the last two terms in
\begin{eqnarray}
       \eta 
& = &  \int_{\cal X} \rho(x,T) \phi(x) \: dx 
     + \int_0^T \int_{\cal X} \rho(x,t) L(x,u) \: dx \: dt \nonumber \\
&   &+ \frac{1}{2} \int_{\cal X} \left\| \frac{\partial u}{\partial x} 
	\right\|^2 dx
     + \frac{1}{2} \int_0^T \left\| \frac{\partial u}{\partial t} 
       \right\|^2 dt.\label{eqn:rholdudxdudt}
\end{eqnarray}
There is in fact a natural stochastic optimal control interpretation to
optimizing (\ref{eqn:rholdudxdudt}) subject to~(\ref{eqn:Liouville_evol}):
Equation~(\ref{eqn:Liouville_evol})~can be identified with the Fokker-Planck
equation corresponding to the stochastic differential equation
$dx = f(x,u) \: dt + G \: dw$ (with $G$ assumed zero); given a probability
density $\rho(x,0) = \rho_0(x)$ at $t=0$, $\rho(x,t)$ then describes the
time evolution of the density for $x$.
	
Adopting the Liouville equation representation, the minimum attention optimal
control problem can now be formulated as follows. Given a desired terminal
density $\psi(x)$, define the attention functional (for convenience we 
henceforth omit the terminal $\phi(x(T))$ term from our functional) $\eta$ as
\begin{eqnarray}
      \eta 
& = & \int_{\mathcal X} \left(\rho(x,T)-\psi(x) \right)^2 dx 
      \nonumber \\
&   &+
	 \int_{\mathcal X} \int_0^T  
	\left(  \left\| \frac{\partial u}{\partial x} \right\|^2 
	      + \left\| \frac{ \partial u}{\partial t} \right\|^2 
        \right) \:dt \: dx,
\label{eqn:attn_cost}
\end{eqnarray}
in which the first and second terms of $\eta$ can be identified as the
terminal cost and running cost, respectively. In this setting the objective
is to seek the $u(x,t)$ that minimizes the attention functional
(\ref{eqn:attn_cost}) while driving the probability density governed
by (\ref{eqn:Liouville_evol}) from $\rho(x,0) = \rho_0(x)$ at $t=0$
to $\rho(x,T) = \psi(x)$ at $t=T$.

Equations~(\ref{eqn:Liouville_evol}) and (\ref{eqn:attn_cost}) can be
compared to the boundary velocity control problem in \cite{Gunzburger},
in which the Navier-Stokes equations and the associated cost functional
are as follows:
\begin{equation}
\left. \begin{array}{ll}
&\frac{\partial v}{\partial t} + (v \cdot \nabla) v - \nu \Delta v + \nabla p = 0~ {\rm in}~ (0,T) 
	\times \Omega, \label{Navier-Stokes}  \\
&\nabla \cdot v = 0~~~~{\rm in}~ (0,T) \times \Omega,  \\
&v = u~~~~ {\rm on}~ (0,T) \times \Gamma_c,  \\
&v = 0~~~~ {\rm on}~ (0,T) \times (\Gamma \backslash \Gamma_c),
\end{array} \right.
\end{equation}
where $v(x,t)$ is a two-dimensional flow over the time interval $[0,T]$ in
the bounded domain $\Omega$ with boundary $\Gamma$, $\Gamma_c \subset \Gamma$
denotes the region over which the control is valid, $p$ and $\nu$ 
respectively denote the hydrodynamic pressure and kinematic viscosity of
the fluid, and $\nabla$ and $\Delta$ respectively the gradient and Laplacian
operators. The initial velocity $v(x,0) = v_0(x)$ is given, and the
boundary velocity control, denoted by $u$, is required to satisfy the
compatibility conditions
\begin{equation}
\int_{\Gamma_c} \langle u, n \rangle \: dx = 0,
\label{boundary-constraint}
\end{equation}
where $n$ denotes the unit outward normal vector on $\Gamma$, and
\begin{equation}
\left. u \right|_{t=0} = \left. v_0 \right|_{\Gamma_c}.
\end{equation}
The associated objective is to find a boundary control $u$ and a velocity
field $v$ such that the cost functional 
\begin{eqnarray}
\mathcal{J}(v, u) & = & 
\int_0^T \int_{\Gamma_c} (\|u\|^2 + \alpha \left\|
\frac{\partial u}{\partial t} \right\|^2
+ \alpha \left\| \frac{\partial u}{\partial x} \right\|^2) \:dx\:dt \nonumber\\
& + & \int_0^T \int_{\Omega} \| v - V \|^2 dx \: dt
\end{eqnarray}
is minimized subject to $v$ and $u$ satisfying 
(\ref{Navier-Stokes})-(\ref{boundary-constraint}) with the prescribed
fluid velocity $V$ and the regularizing parameter $\alpha$.

Although the minimum attention problem seeks a control over the entire
domain rather than on the boundary, and the domain is also of arbitrary
dimension $n$ rather than two-dimensional, the one-shot method used in
\cite{Gunzburger} for the boundary flow control problem can be generalized
to prove both existence of a solution and convergence of an iterative
gradient descent algorithm for the minimum attention problem.

As a first step, the first-order necessary conditions for
$(\ref{eqn:attn_cost})$ subject to the state dynamics
$(\ref{eqn:Liouville_evol})$ can be derived explicitly in terms of
the augmented Lagrangian $\mathscr{L}$:
\begin{equation}
\mathscr{L} = \eta - \int_{\mathcal X} \int_0^T \lambda \left(
\frac{\partial{\rho(x,t)}}{\partial t} + \langle
\frac{\partial}{\partial x}, f(x,u) \rho(x,t) \rangle \right) \: dt \: dx,
\label{eqn:Lagrangian}
\end{equation}
where $\lambda$ denotes the Lagrange multiplier.  Then
$\frac{\delta{\mathscr{L}}}{\delta \lambda} = 0$ leads to the state
equations~(\ref{eqn:Liouville_evol}), while the adjoint equation
$\frac{\delta{\mathscr{L}}}{\delta \rho} = 0$ 
leads to
\begin{equation}
  \frac{\partial \lambda(x,t)}{\partial t} 
 +\left< \frac{\partial \lambda(x,t)}{\partial x}, f(x,u) \right> = 0,
\label{eqn:advection}
\end{equation}
subject to the boundary condition $\lambda = \rho-\psi$ at $t=T$.  Finally,
the optimality condition $\frac{\delta{\mathscr{L}}}{\delta u}=0$ leads to
\begin{equation}
\Delta u + \frac{\partial^2 u}{\partial t^2}
= \frac{\partial f}{\partial u}^T 
  \frac{\partial \lambda}{\partial x} \rho.
\label{eqn:optimality_eq}
\end{equation}
Equations~(\ref{eqn:Liouville_evol}), (\ref{eqn:advection}),
(\ref{eqn:optimality_eq}) together constitute the first-order necessary
conditions for the optimal solution $(\rho, u, \lambda)$.  Since $\lambda$
is conserved along the direction of advection $f$, the Lagrange multiplier
$\lambda(x(t), t)$ along solutions to (\ref{eqn:xdotfxu}) can be expressed
in terms of the terminal densities $\rho$ and $\psi$ as
\begin{eqnarray}
      \lambda(x(t),t) 
& = & \lambda \left( x(t)+\int_t^T f \: dt, T \right) \nonumber \\
& = & \rho(x(t)+\int_t^T f \: dt,T)-\psi(x(t)+\int_t^T f \: dt) \nonumber \\
& = & \rho(x(T),T) - \psi(x(T)) \label{formula:lambda}.
\end{eqnarray}
We make use of these first-order necessary conditions in the next section.


\section{Main Result}
\label{sec:proof_and_iterative_algorithm}

We now assume that the control $u$ is of the form
\begin{equation}
u = K(t)x + v(t),
\label{input_case2}
\end{equation}
where $K(t) \in \mathbb{R}^{m \times n}$ and $v(t) \in \mathbb{R}^{m}$ are
both $C^2([0, T])$, i.e., twice continuously differentiable on $[0,T]$.

\begin{proposition}
Given an initial density $\rho_0(x)$ and terminal density $\psi(x)$
over ${\cal X}$ bounded, there exists a sequence $(K_i, v_i)$, with 
$K_i(t) \in \mathbb{R}^{m \times n}$ and $v_i(t) \in \mathbb{R}^m$
both $C^2([0,T])$ and equicontinuous, that uniformly converges to
some $(\hat{K}, \hat{v})$ in the space $C^2([0,T])$, such that
$\hat{u}=\hat{K}(t)x + \hat{v}(t)$ minimizes the functional $\eta$ 
in (\ref{eqn:attn_cost}).
\end{proposition}

\noindent
\begin{proof}
By continuity $K$ and $v$ are both bounded. $\eta$ is also bounded from
above and below, i.e., $\eta_{\rm m} \leq \eta \leq \eta_{\rm M}$ where 
\begin{eqnarray}
\eta_{\rm m} & = & \inf_{K, v \in C^2([0,T])} \eta(K, v),\\
\eta_{\rm M} & = & \sup_{K, v \in C^2([0,T])} \eta(K, v).
\end{eqnarray}
Let $K_n$ and $v_n$ be minimizing sequences for $\eta$, i.e.,
\begin{equation}
\lim_{n \rightarrow \infty} \eta(K_n, v_n) = \eta_{\rm m},
\end{equation}
equicontinuous in the class of $C^2([0,T])$. Since both $K_n$ and
$v_n$ are pointwise bounded and equicontinuous, there exists a
subsequence $(K_m, v_m)$ in $C^2([0,T])$ uniformly converging to
$(\hat{K}, \hat{v})$ \cite{Rudin}. ($\hat{K}$, $\hat{v}$) and the
corresponding $\hat{u}$ minimize the functional $\eta$ of
(\ref{eqn:attn_cost}).
\end{proof}


\subsection*{Iterative Algorithm}
\label{sec:algorithm}

Having established the existence of a solution, we now develop an iterative
gradient descent algorithm (Algorithm \ref{algorithm1}) to solve the
first-order necessary conditions for the optimal $(K(t), v(t))$.  The control
$u$ in the form of (\ref{input_case2}) is substituted into the optimality
condition (\ref{eqn:optimality_eq}). Keeping in mind that solutions obey
(\ref{eqn:xdotfxu}), the explicit equation to be solved is, after
some calculation,
\begin{eqnarray}
(\ddot{K} + K B \dot{K}) x(t) = \rho(x(t), t) B^T
\frac{\partial \lambda(x(t), t)}{\partial x}  \nonumber \\ 
-\ddot{v} - 2\dot{K}f
- K( Af  + BKf + B\dot{v}),
\label{case2_eq3}
\end{eqnarray}
where $A = \frac{\partial f}{\partial x}, B = \frac{\partial f}{\partial u}$,
and $x(t)$ is governed by (\ref{eqn:xdotfxu}) with initial condition 
$x(0) = x_0$.  The associated $\rho(x,t)$ follows the Liouville equation
(\ref{eqn:Liouville_evol}) with initial condition $\rho(x,0) = \rho_0(x)$
while $\lambda(x(t),t)$ is obtained from (\ref{formula:lambda}).

Referring to Algorithm \ref{algorithm1}, for initialization we generate an
initial reference trajectory and control $(x^*,u^*)$ using, e.g., the linear
quadratic regulator method of \cite{Li}. The initial time-varying feedback
gain and feedforward $(K^{(0)}(t), v^{(0)}(t))$ are then derived from
the linearized optimal control perturbed from $(x^*,u^*)$ \cite{Jang}:
\begin{equation}
u(x,t) = u^*(t) - R^{-1} B(t)^{T} P(t)(x-x^*(t)),
\end{equation}
where $R \in \mathbb{R}^{m \times m}$ is a symmetric positive-definite
matrix, and $P(t)  \in \mathbb{R}^{n \times n}$ is generated
backward from the matrix Riccati differential equation
\begin{equation}
-\dot{P} = PA + A^TP - PBR^{-1}B^T P,
\end{equation}
with terminal boundary condition $P(T) = P_{\rm f}$.  $K^{(0)}$ and
$v^{(0)}$ are prescribed as follows:
\begin{eqnarray}
K^{(0)} &=& -R^{-1} B^{T} P, \label{K_initial}\\
v^{(0)} &=& u^* + R^{-1} B^{T} Px^*. \label{v_initial}
\end{eqnarray}
Given $K^{(0)}(t)$ and $v^{(0)}(t)$ from the optimal control, we solve for
the density distribution $\rho^{(0)}(x,t)$ and evaluate the cost $\eta^{(0)}$.

We note that our algorithm does not require the initial path to satisfy the
terminal endpoint constraint.  For initialization of the line search parameter
$\epsilon_{(m)}^{(n)}$, we set $\epsilon_{(0)}^{(0)} = \epsilon_0$, where
$n$ and $m$ respectively indicate the iterations for the outer loop control
update and inner loop line search.

\begin{algorithm}[b]
\caption{Iterative Gradient Algorithm for $K$ and $v$}
\begin{algorithmic}
\STATE{Initialize $K^{(0)}(t)$, $v^{(0)}(t)$, $\rho^{(0)}(x,t)$,
	and $\eta^{(0)}$}
\STATE{$\epsilon_{(0)}^{(0)} \gets \epsilon_0$}
\REPEAT
\STATE{Solve (\ref{eqn:xdotfxu}) for $x^{(n-1)}$ with
$u=K^{(n-1)}x^{(n-1)}+v^{(n-1)}$}
\REPEAT
\STATE{$\Delta v(t) \gets ( -\rho(x(t),t) B^T
	\frac{\partial \lambda(x(t),t)}{\partial x} + \ddot{v}$}
\STATE{$ ~~~~~~~ + 2 \dot{K}f + K( Af + BKf  + B\dot{v} ) )^{(n-1)}(t) $}
\STATE{$\Delta K(t) \gets (\ddot{K} + KB\dot{K})^{(n-1)}(t)$}
\STATE{$v^{(n)} \gets v^{(n-1)} -\epsilon_{(m)}^{(n)} \Delta v$}
\STATE{$K^{(n)} \gets K^{(n-1)} -\epsilon_{(m)}^{(n)} \Delta K$}
\STATE{Solve (\ref{eqn:Liouville_evol}) for $\rho^{(n)}$ with
$(K^{(n)}, v^{(n)})$ and evaluate $\eta^{(n)}_{(m)}$}
\STATE{ $\epsilon_{(m+1)}^{(n)} \gets 0.5 \epsilon_{(m)}^{(n)}$}
\STATE{$m \gets m+1$}
\UNTIL{$\eta^{n}_m \leq \eta^n_{(m-1)}$}
\STATE{$\epsilon_{(0)}^{(n+1)} \gets 1.5 \epsilon_{(m)}^{(n)}$, 
$\eta^{(n)} \gets \eta^{(n)}_{(m)}$; evaluate $\lambda^{(n)}$}
\STATE{$n \gets n+1$}
\UNTIL{$\frac{|\eta^{(n)} - \eta^{(n-1)}|}{\eta^{(n-1)}} \leq
\epsilon_{\rm tol}$}
\end{algorithmic} \label{algorithm1}
\end{algorithm}
%

We now show that if the second variation of $f$ with
respect to $u$ is degenerate, our algorithm converges to a local minimizer
for any initial control:
\begin{proposition} Algorithm \ref{algorithm1} converges to a local minimizer
for any initial control if the second variation of $f$ with respect
to $u$ is degenerate.
\end{proposition}
\begin{proof}
Following \cite{Gunzburger}, let ${\cal X}$ be a Hilbert space with norm
$\|\cdot\|$ and let $\mathscr{L}$ be a $C^2$ functional on ${\cal X}$.
Suppose $\hat{u}$ is a local minimizer for $\mathscr{L}$ close to the
initial control, and let $B$ be a ball in ${\cal X}$ centered at $\hat{u}$
with radius sufficiently large so as to include the initial control. 
From the optimality condition (\ref{eqn:optimality_eq}),
\begin{equation}
\frac{\delta \mathscr{L}}{\delta u} = -\int_{\mathcal{X}} \int_0^T
\left( \Delta u + \frac{\partial^2 u}{\partial t^2}
- (\frac{\partial f}{\partial u})^T \frac{\partial \lambda}{\partial x}
\rho \right) dt \: dx.
\end{equation}
If the second variation of $f$ with respect to $u$ is degenerate, for
all $\tilde{u} \in B$ and for all variations $h_1, \: h_2 \in X$ of
$\tilde{u}$, from (\ref{eqn:Lagrangian}) the second variation of
$\mathscr{L}$ can be written
\begin{equation}
\frac{\delta^2 \mathscr{L}}{\delta u^2}(\tilde{u})(h_1, h_2)
= -\int_{\cal X} \int_0^{T} (\Delta h_2 + 
\frac{\partial^2 h_2}{\partial t^2}) \cdot h_1 dt \: dx,
\end{equation}
where $\frac{\delta^2 \mathscr{L}}{\delta u^2}(\tilde{u})(h_1, h_2)$
is the bilinear form related to the second derivatives of $\mathscr{L}$.
Because $h_2$ is a variation of $\tilde{u}$ that is linear in $x$ as assumed 
in~(\ref{input_case2}), $\nabla^2 h_2$ is degenerate. Integrating by parts,
\begin{equation}
\int_0^{T} \frac{\partial^2 h_2}{\partial t^2} \cdot h_1 \: dt
= -\int_0^{T} \frac{\partial h_2}{\partial t} \cdot
\frac{\partial h_1}{\partial t} dt.
\end{equation}
Taking the time derivative of $h=K \: \delta x$ as a variation of $\tilde{u}$,
\begin{equation}
\frac{\partial h}{\partial t} = (\dot{K} + KA) \delta x.
\end{equation}
Define the constants $c_1$ and $c_2$ to be the supremum and infimum of
the ratio $\| \frac{\partial h}{\partial t} \| / \| h \|$, i.e.,
\begin{equation}
c_1 = \sup_{K \delta x \neq 0} \frac{\| (\dot{K} + KA) \delta x \|}
	{\| K \delta x \|}, ~~
c_2 = \inf_{K \delta x \neq 0} \frac{\| (\dot{K} + KA) \delta x \|}
	{\| K \delta x \|}.
\end{equation}
Because $K$, $\dot{K}$, and $A$ are all continuous and $\delta x$ is
bounded in $\cal X$, the existence of a positive finite $c_1$ and $c_2$
are guaranteed. The second variation of $\frac{\delta^2 \mathscr{L}}{\delta
u^2}(h_1, h_2)$ is therefore bounded from above and below, i.e.,
\begin{eqnarray}
\frac{\delta^2 \mathscr{L}}{\delta u^2}(\tilde{u})(h_1, h_2)
& \leq & c_1^2T ||{\cal X}|| ||h_1|| ||h_2||
\label{second_variation_f_lower_bound}\\
\frac{\delta^2 \mathscr{L}}{\delta u^2}(\tilde{u})(h_1, h_1) 
& \geq & c_2^2 T ||{\cal X}|| ||h_1||^2,
\label{second_variation_f_upper_bound}
\end{eqnarray}
satisfying ellipticity of the Lagrangian. It is now possible to choose a
positive set of parameters $\epsilon_k < (2c_2^2/c_1^2)$ such that the
iteration
\begin{equation}
u(k+1) = u(k) - \epsilon_k \nabla \mathscr{L}(u(k)), k = 0, 1, 2, ...
\end{equation}
converges to $\hat{u}$ for any initial iterate $u(0) \in  B$.  For further
related details see \cite{Ciarlet}.
\end{proof}


\section{Example}
\label{sec:example}
As an example, we derive the minimum attention control law for a two-link
planar robot arm. The dynamic equations are of the form
\begin{equation}
\tau= M(q) \ddot{q} + b(q, \dot{q})
\end{equation}
where $q \in \mathbb{R}^2$ represents the joint angle vector and $\dot{q} \in
\mathbb{R}^2$ its velocity, $\tau \in \mathbb{R}^2$ is the joint torque vector,
$M(q) \in \mathbb{R}^{2 \times 2}$ is the symmetric positive-definite mass
matrix, and $b(q, \dot{q}) \in \mathbb{R}^2$ consists of the Coriolis and
gravity terms.  Expressions for $M(q)$ and $b(q, \dot{q})$ can be found
in Appendix \ref{appendix1}.  Define the state vector $x=(q, \dot{q}) \in
\mathbb{R}^4$ and the control $u = \tau \in \mathbb{R}^2$. The dynamic
equations in state space form then become
\begin{eqnarray}
\dot{x}=f(x,u)= \left[ \begin{array}{c}
x_2\\
M^{-1}(x)\left( u-b(x_1,x_2) \right)
\end{array} \right].
\label{two-link-dynamics}
\end{eqnarray}
In this example, $u = K(t)x+v(t)$ is bounded with $K(t)$
and $v(t)$ continuous in $[0,T]$. Since $f = f(x,u)$ is continuously
differentiable, $(q, \dot{q})$ is bounded in its trajectory, and
and the second variation of $f$ with respect to $u$ is degenerate,
our minimum attention algorithm is applicable. 

We perform two sets of experiments. In the first experiment, the objective
is to drive the system from the initial state $x_{\rm init}=(0,0,0,0)$ to
the end-effector position-velocity $\phi_{\rm f}=(-0.26, 0.40, 0, 0)$ at
terminal time $T=0.5$ sec (the forward kinematics $\phi(x)$ is described
in Appendix \ref{appendix2}).  In the second experiment we drive the system
from the initial state $x_{\rm init}=(0.1,-0.1,0,0)$ to the end-effector
position-velocity $\phi_{\rm f}=(-0.32, 0.27, 0, 0)$ at terminal
time $T=0.5$ sec. The associated initial distribution $\rho_0(x)$
is prescribed to be centered at $x_{\rm init}$, with compact support
over eight grids using a smoothed Dirac delta function (see \cite{Peskin}).
The attention cost (\ref{eqn:attn_cost}) for this example is chosen to be
\begin{eqnarray}
\eta &=& \gamma \int_{\mathcal X} 
 \| \phi(x) - \phi_{\rm f} \|^2 \rho(x,T) \: dx \\
&+& \int_{\mathcal X} \int_0^T 
\left\| \frac{\partial u}{\partial x} \right\|^2
      + \left\| \frac{ \partial u}{\partial t} \right\|^2
\:dt \: dx. \nonumber
\label{eqn:attn_cost_ball_catching}
\end{eqnarray}
The adjoint boundary condition is $\lambda = \gamma \| \phi(x)
- \phi_{\rm f} \|^2$ at $t=T$, with $\lambda$ at $(x(t), t)$
derived similarly to (\ref{formula:lambda}):
\begin{equation}
\lambda(x(t),t) = \gamma\| \phi(x(T)) - \phi_{\rm f} \|^2.
\end{equation}
For our example we choose $\gamma = 10^{6}$.  The state space $\cal X$
is taken to be $[-5,5] \times [-5,5] \times [-300, 300] \times [-300,300]$,
discretized into 256 uniform intervals over each dimension, while the time
domain is discretized into 40 intervals. 

For Experiment 1, the initial $K^{(0)}(t)$ and $v^{(0)}(t)$ are obtained
from (\ref{K_initial}) and (\ref{v_initial}) with $P_{\rm f}
={\rm diag}(10^5,1,10^5,1)$ and $R={\rm diag}(0.4,1.3565)$.
These $K^{(0)}(t)$ and $v^{(0)}(t)$ are also used for Experiment
2 with the same set of parameters. For the derivatives of $K$ and $v$, 
second-order centered differencing is used except for the initial and
terminal points at $t=0$ and $t=T$, where first-order differencing
is employed. The second derivatives of $K$ and $v$ are similarly computed
from their first derivatives. We set $\epsilon_0=2.0 \times 10^{-3}$
and $\epsilon_{\rm tol}=5.0 \times 10^{-5}$ in the main algorithm.
Also, at each iteration of the main algorithm, for the evaluation of
$\rho^{(i)}(x,t)$ we use the basic Monte Carlo Algorithm
\ref{Monte_Carlo_method} \cite{Bird} (but other methods are also possible).

\begin{algorithm}[bt]
\caption{Monte Carlo evaluation of $\rho(x,t)$ given $(K,v)$}
\begin{algorithmic}
\STATE{trackcount $\gets$ 0}
\STATE{Generate grids with height $\Delta x$ and width $\Delta t$ at nodes
	$(x_i, t_i)$} 
\WHILE{trackcount $<$ trackmax}
\STATE{Sample $x_{\rm r} \in {\cal X}_c$ from a uniform distribution}
\STATE{Sample $r \in [0,1]$ from a uniform distribution}
\IF {$r < \rho_0(x_{\rm r}) \Delta x \Delta t$}
\STATE{Generate $x(t)$ from (\ref{eqn:xdotfxu}) with $x_0 = x_{\rm r}$
and $u=Kx+v$}
\STATE{Increment $N_{\rm pass}(x_i, t_i)$ when $x(t)$ passes through the grid
centered at $(x_i, t_i) \subseteq {\cal X} \times [0,T]$}
\STATE{trackcount $\gets$ trackcount + 1}
\ENDIF
\ENDWHILE
\STATE{$\rho(x_i,t_i) \gets N_{\rm pass}(x_i, t_i) / {\rm trackmax}$}
\end{algorithmic}
\label{Monte_Carlo_method}
\end{algorithm}

Figure 1(a) shows the results of Experiment 1, in which the initial
path in $q_1$-$q_2$ space is shown in red, and the minimum attention
path is shown in blue. Note that although the initial path does not
reach the target goal state, our algorithm converges to a final path
that reaches the target. Figure 1(b) shows snapshots of the robot
arm motion as it traverses the minimum attention path from the initial
to final pose. 
Figure 2 shows trajectories for the minimum attention control $u(t)$,
feedback gain $K(t)$, and feedforward term $v(t)$.  Consistent with
observations from the human motor control literature \cite{Jang,
Kawato, Wolpert2, Liu}, we also observe that the feedforward term 
initially dominates in the early part of the motion but eventually
approaches zero, while the feedback term initially starts out small
but rapidly dominates toward the latter part of the motion, when
feedback is used for precise positioning. The attention cost plotted
against the number of iterations is shown in Figure 3.

Results of Experiment 2 are shown in Figure 4, which shows the initial
and final paths in $q_1$-$q_2$ and the corresponding arm pose snapshots
for a different pair of initial and terminal states.  The optimal
trajectories obtained for $u(t)$, $K(t)$, and $v(t)$ also show
similar characteristics to those for Experiment 1, i.e., a smooth
transition from feedforward to feedback terms as the motion progresses
from the initial to final state.

\begin{figure}[hbt]
\begin{center}
$\begin{array}{c@{\hspace{0.05in}}c@{\hspace{0.05in}}c}
\includegraphics[angle=0, width=0.4\textwidth]{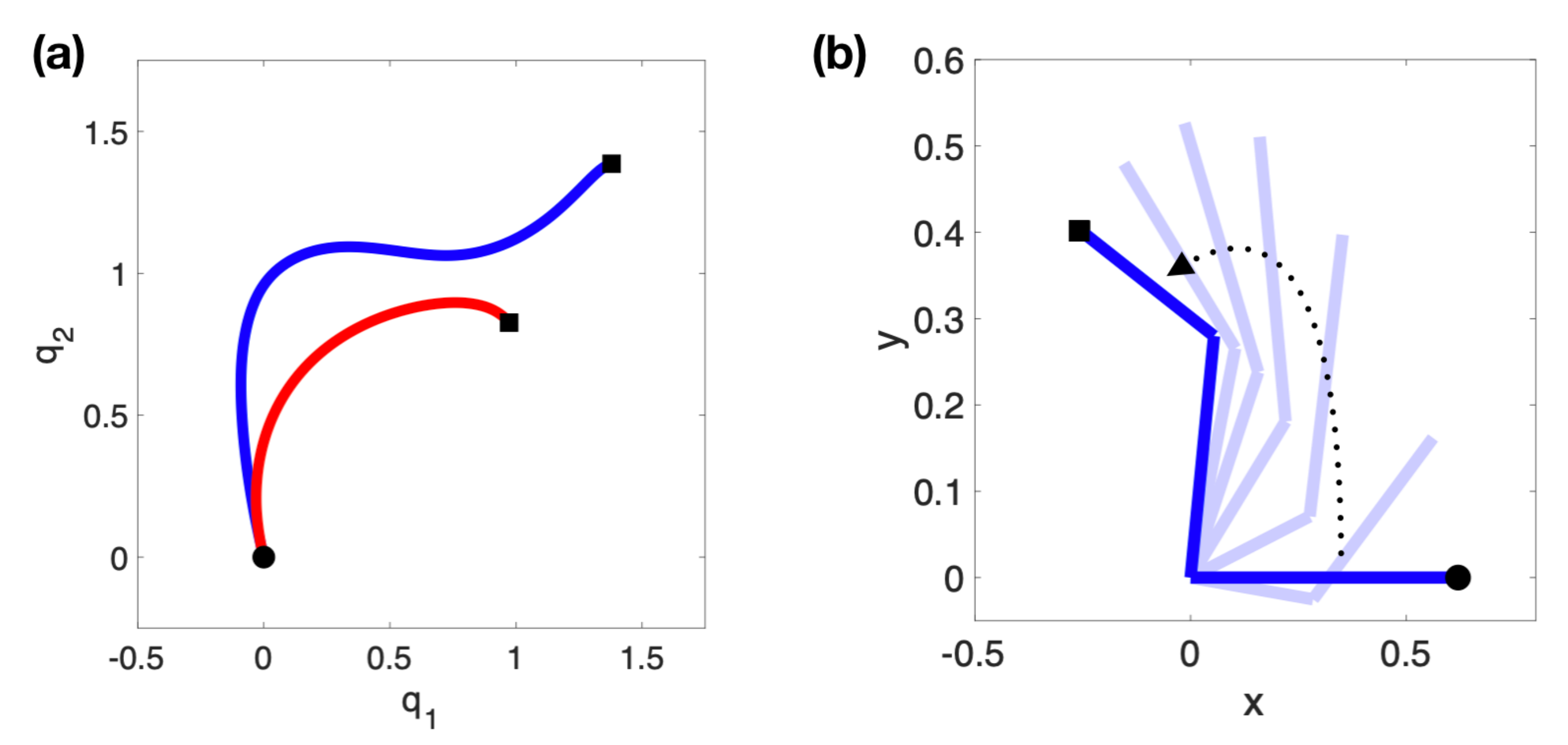}
\end{array}$
\end{center}
\caption{Experiment 1: the robot is driven from the initial configuration
$x_{\rm init}=(0,0,0,0)$ to the end-effector position-velocity
$\phi_{\rm f}=(-0.26, 0.40, 0, 0)$ at terminal time $T=0.5$ sec.
(a) Initialized (red), converged (blue) paths in $q_1$-$q_2$ with
the initial state (circle) and the final state (box).
(b) A sequence of arm pose snapshots for the final converged motion.
The initial and final hand positions are indicated by the 
circle and box, respectively.}
\label{figure1}
\end{figure}

\begin{figure}[hbt]
\begin{center}
$\begin{array}{c@{\hspace{0.05in}}c@{\hspace{0.05in}}c}
\includegraphics[angle=0, width=0.4\textwidth]{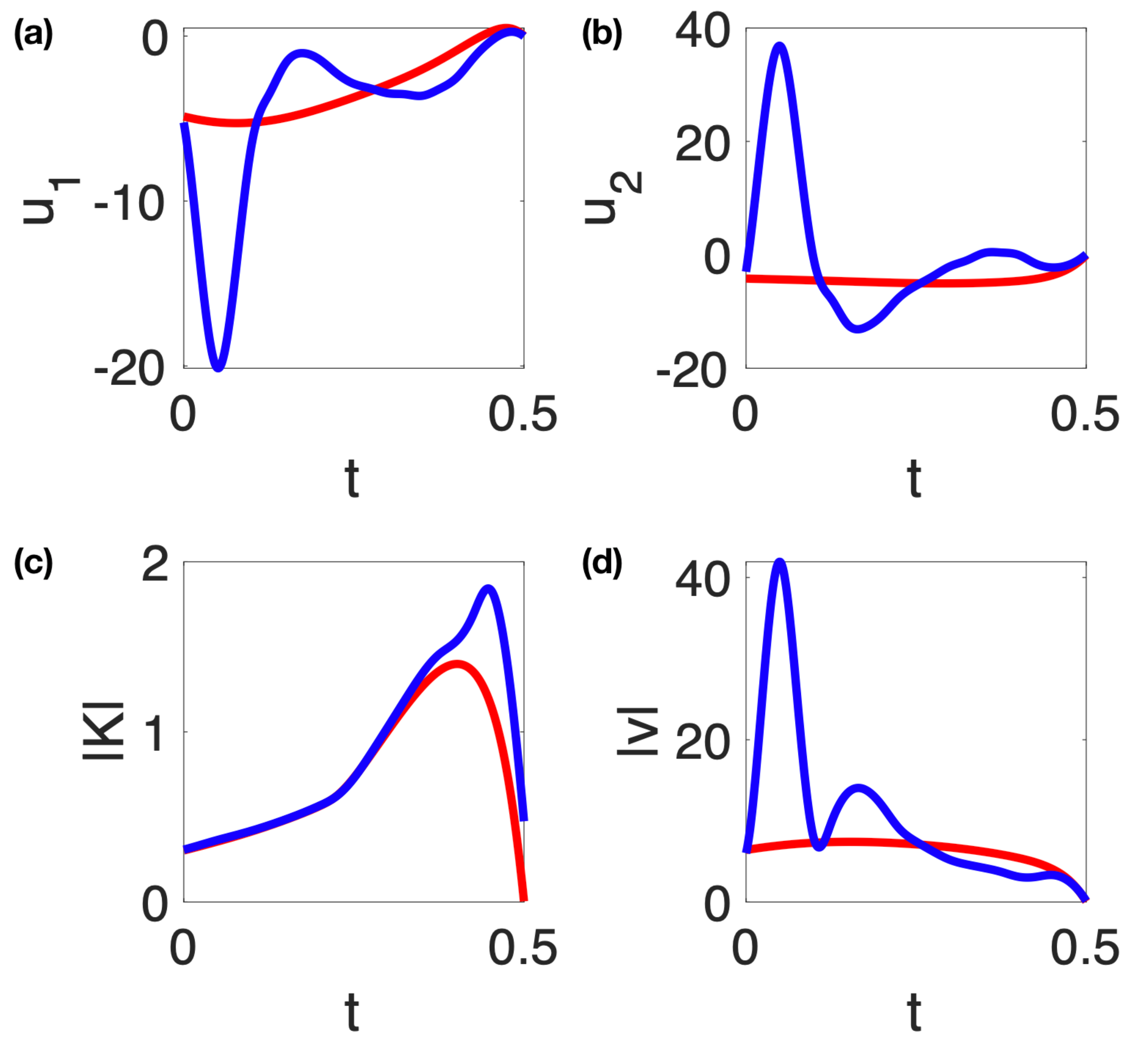}
\end{array}$
\end{center}
\caption{Experiment 1: Graphs of (a) $u_1(t)$; (b) $u_2(t)$;
(c) $\|K(t)\|$; (d) $\|v(t)\|$. Initial trajectories are indicated 
in red, while converged trajectories are indicated in blue.}
\label{figure2}
\end{figure}

\begin{figure}[hbt]
\begin{center}
$\begin{array}{c@{\hspace{0.05in}}c@{\hspace{0.05in}}c}
\includegraphics[angle=0, width=0.3\textwidth]{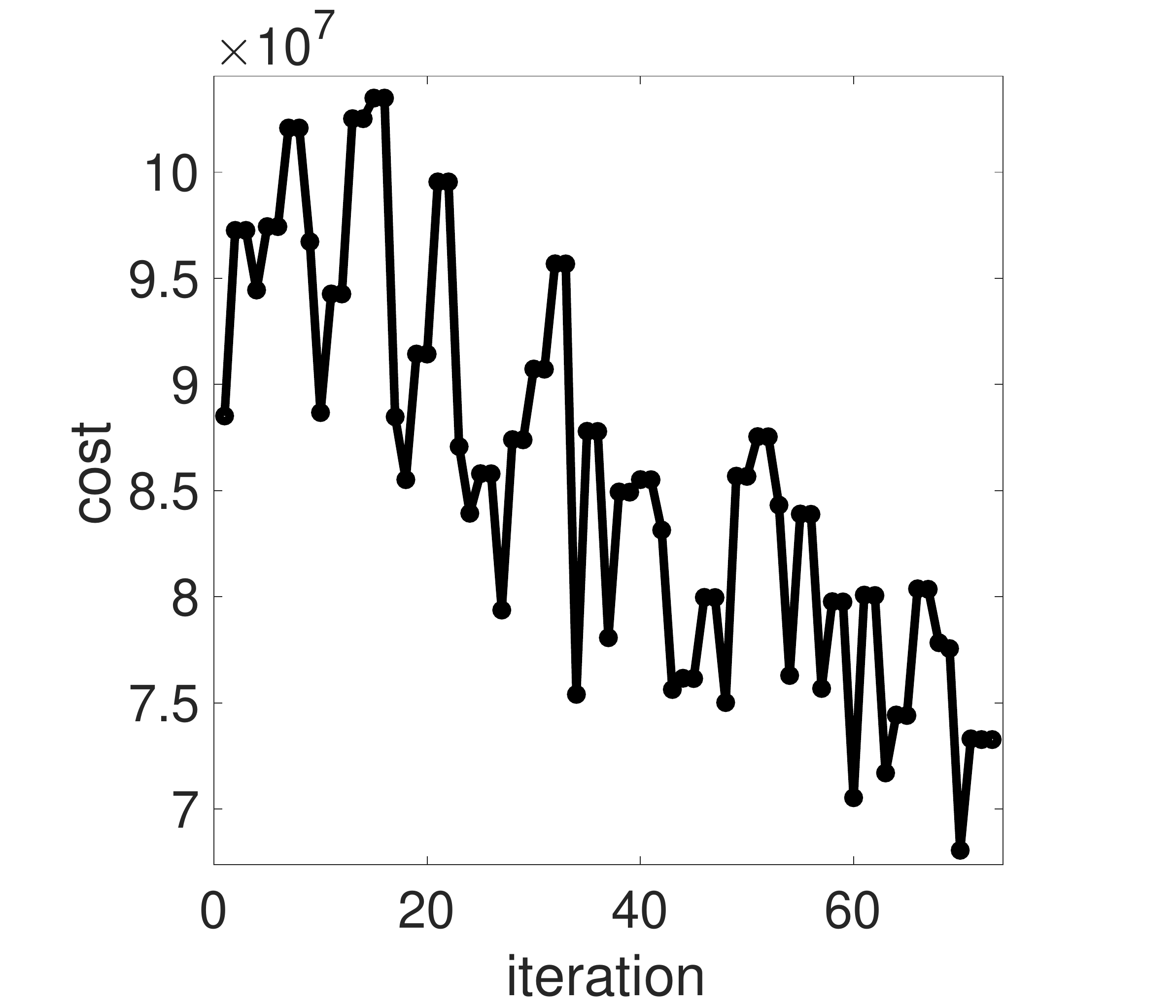}
\end{array}$
\end{center}
\caption{Experiment 1: Attention cost $\eta$ plotted against 
the number of iterations.}
\label{figure3}
\end{figure}

\begin{figure}[hbt]
\begin{center}
$\begin{array}{c@{\hspace{0.05in}}c@{\hspace{0.05in}}c}
\includegraphics[angle=0, width=0.4\textwidth]{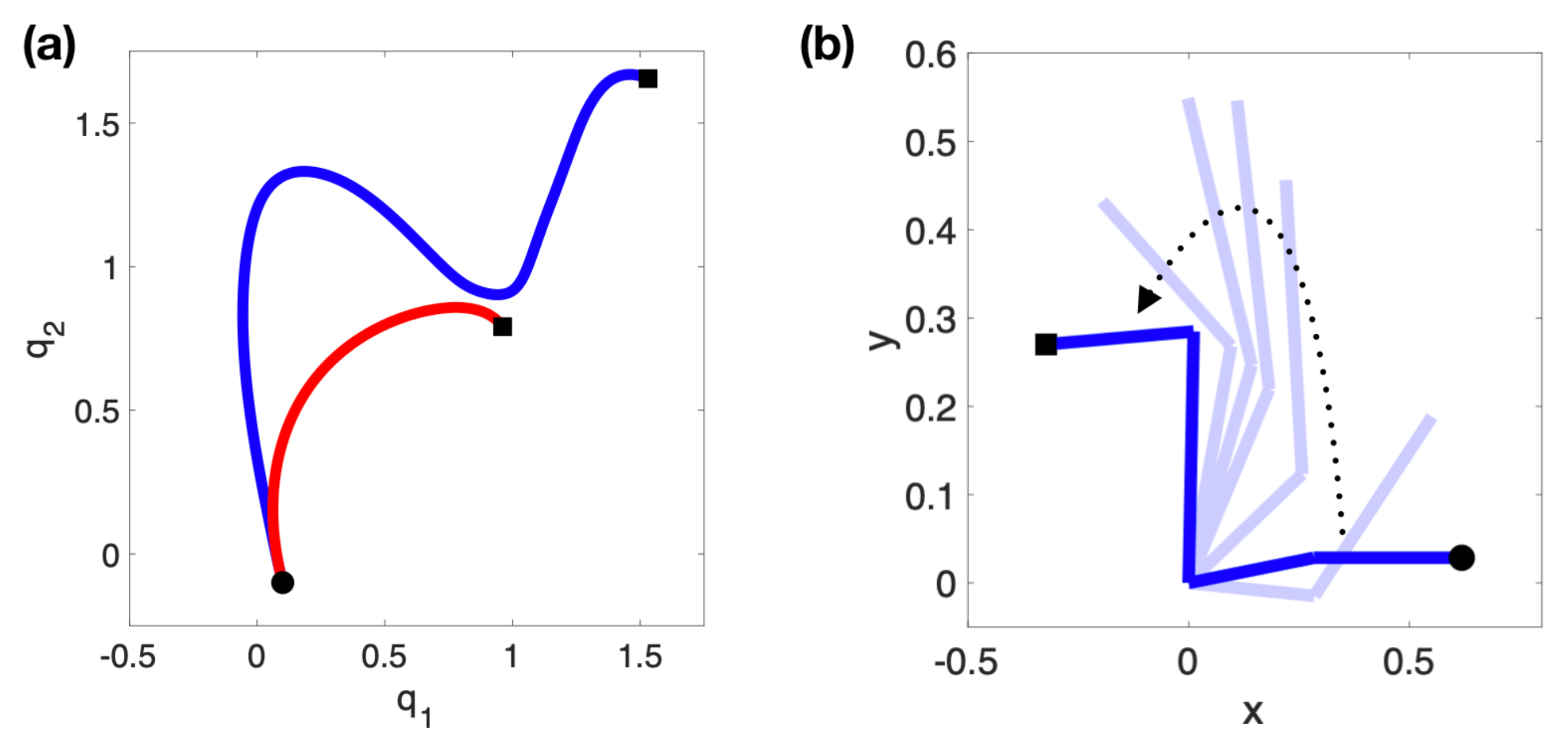}
\end{array}$
\end{center}
\caption{Experiment 2: the robot is driven from the initial configuration
$x_{\rm init}=(0.1,-0.1,0,0)$ to the end-effector position-velocity
$\phi_{\rm f}=(-0.32, 0.27, 0, 0)$ at terminal time $T=0.5$ sec. 
(a) Initialized (red), converged (blue) paths in $q_1$-$q_2$ with
the initial state (circle) and the final state (box).
(b) A sequence of arm pose snapshots for the final converged motion.
The initial and final hand positions are indicated by the 
circle and box, respectively.}
\label{figure4}
\end{figure}


\section{Conclusion}

Using the Liouville partial differential equation representation for
nonlinear systems \cite{Brockett2}, in this paper we have provided a proof
of the existence of a minimum attention control as formulated in
\cite{Brockett} under the assumption of a control of the form $u(x,t)
= K(t) x + v(t)$.  Exploiting the structural similarity between this
problem and a boundary flow control problem involving the Navier-Stokes
equations, we adopt the one-shot method of \cite{Gunzburger} to develop
an iterative gradient algorithm for the numerical solution of the
minimum attention control. An example involving a two-link planar
robot arm is used to illustrate the algorithm, with results that bear
close similarity to observed human arm reaching movements that transition
from feedforward to feedback control as the goal state is reached.  Our
algorithm, although not yet sufficiently efficient for real-time control
applications, reliably converges to a local minimum for a wide range of
initial conditions.  The possibility of using our algorithm to generate
training data for, e.g., a reinforcement learning-based method for motion
control \cite{Levine} is but one possible practical application.


\section*{Acknowledgments}
F.C.~Park was supported in part by SNU SRRC grant NRF-2016R1A5A1938472,
SNU BMRR grant DAPAUD190018ID, the Naver Labs AMBIDEX Project,
SNU BK21+ Program in Mechanical Engineering, SNU-IAMD, and the
SNU Institute for Engineering Research.


\appendices
\section{Dynamics of two-link robot arm}
\label{appendix1}
Elements of the mass matrix $M(q)$ and bias terms $b(q, \dot{q})$ are as
follows \cite{Nakano}:
\begin{eqnarray*}
m_{11} &=&  I_1 + I_2 + 2 M_2 L_1 S_2 \cos q_2 + M_2 L_1^2 \\
m_{22} &=&  I_2 \\
m_{12} &=&  I_2 + M_2 L_1 S_2 \cos q_2 \\
m_{21} &=&  m_{12} \\
b_1    &=& -M_2 L_1 S_2 (2 \dot{q}_1+\dot{q}_2) q_2 \sin q_2+B_{11}\dot{q}_1
           +B_{12} \dot{q}_2 \\
       & & +g[(M_1 S_1+M_2 L_1) \sin q_1 +M_1 S_2 \sin (q_1 + q_2)] \\
b_2    &=&  M_2 L_1 S_2 \dot{q}_1^2 \sin q_2+B_{22} 
	    \dot{q}_2+B_{21}\dot{q}_1 \\
       & & +g M_2 S_2 \sin(q_1 + q_2),
\end{eqnarray*}
where $g$ denotes gravity, and $I_i$, $M_i$, $L_i$, $S_i$ respectively
denote the inertia, mass, arm length, and center of mass of link $i$.

\section{Forward Kinematics}
\label{appendix2}
The forward kinematics map $\phi(x)$ maps the joint position-velocity pair
$x=(q, \dot{q})$ to the end-effector Cartesian position $(X,Y)$ and velocity
$(\dot{X}, \dot{Y})$ via
\begin{equation*}
\left( \begin{array}{c} X\\ Y\\ \dot{X} \\ \dot{Y} \end{array} \right) =
\left( \begin{array}{c}
 L_1 \cos q_1 + L_2 \cos (q_1 + q_2)\\
 L_1 \sin q_1 + L_2 \sin (q_1 + q_2)\\
-L_1 \dot{q}_1 \sin q_1 - L_2 (\dot{q}_1+\dot{q}_2) \sin(q_1 + q_2)\\
 L_1 \dot{q}_2 \cos q_1 + L_2 (\dot{q}_1+\dot{q}_2) \cos(q_1 + q_2)
\end{array} \right).
\end{equation*}

\ifCLASSOPTIONcaptionsoff
  \newpage
\fi

\end{document}